\documentclass[11pt]{article}

\usepackage[T1]{fontenc}
\usepackage[utf8]{inputenc}
\usepackage{lmodern}
\usepackage{amsmath,amssymb,amsthm,mathtools}
\usepackage[a4paper,margin=2.6cm]{geometry}
\usepackage{microtype}
\usepackage{enumitem}
\usepackage[hidelinks]{hyperref}

\newtheorem{theorem}{Theorem}[section]
\newtheorem{lemma}[theorem]{Lemma}
\newtheorem{proposition}[theorem]{Proposition}
\newtheorem{corollary}[theorem]{Corollary}
\theoremstyle{remark}
\newtheorem{remark}[theorem]{Remark}

\newcommand{\Pp}{\mathbb P}
\newcommand{\Ee}{\mathbb E}
\newcommand{\R}{\mathcal R}
\newcommand{\cE}{\mathcal E}
\newcommand{\1}{\mathbf 1}
\newcommand{\wh}{\widehat}
\newcommand{\Root}{\mathcal R_q}
\newcommand{\dd}{\,\mathrm d}
\newcommand{\norm}[1]{\left\lVert #1\right\rVert}

\title{Punctured Wilson--Evolving Sets and Root Identities for Massive Kirchhoff Forests}
\author{Nordine Anis Moumeni\\[0.4em]
\small Aix-Marseille Universit\'e, CNRS, Centrale Marseille, I2M, UMR 7373, Marseille, France\\
\small\texttt{nordine.moumeni@univ-amu.fr}}
\date{}

\begin{document}
\maketitle

\begin{abstract}
Let $F_q$ be the rooted forest sampled by Wilson's algorithm for a finite reversible
Markov chain killed at rate $q$. We express the probability that $k$ prescribed
vertices are roots as a product. For the $i$th factor, start an evolving set from the singleton consisting of
the $i$th vertex, in the domain obtained by removing the preceding $i-1$ vertices.
The factor is the probability that the set contains its starting vertex at an
independent exponential time of rate $q$. These factors are the successive
Schur-complement pivots of the determinantal root kernel.
The construction couples the walk up to its exit from a domain with the evolving set
biased by its mass. For the ordinary evolving set, the expected mass at that time divided by the
initial mass is the probability that the walk has not yet left the domain.
We use the probability that a second walk hits the first Wilson branch to bound
the probability that two vertices belong to the same component. For a symmetric
transition kernel, the expected size of the component of any vertex is at most
$(1+1/q)^2$. Under a Gaussian heat-kernel upper bound with exponent $\nu>4$, we also
obtain spatial decay on the scale
$q^{-1/2}$, up to a finite-volume correction.
\end{abstract}

\section{Introduction and main results}

Let $V$ be a finite set and let $P$ be an irreducible Markov kernel on $V$, reversible
with respect to a strictly positive measure $\pi$.

Set $L=I-P$. Write $(X_t)_{t\geq0}$ for the continuous-time chain with
generator $-L$ and semigroup $P_t=e^{-tL}$. Let $q>0$. Throughout the paper,
$T_q$ denotes an exponential time of rate $q$, independent of the processes under
consideration.

For $x,y\in V$, write $c(x,y)=\pi(x)P(x,y)=\pi(y)P(y,x)$. Let $F$ be a
rooted spanning forest and let $\R(F)$ be its set of roots. We assign $F$ probability
\begin{equation}\label{eq:forest-law}
  \Pp_q(F)
  =\frac{1}{Z_q}
  \left(\prod_{\{x,y\}\in F}c(x,y)\right)
  \left(\prod_{r\in\R(F)}q\pi(r)\right).
\end{equation}
To see the connection with Wilson's algorithm, orient each tree towards its root and
write $p_F(x)$ for the parent of a non-root vertex $x$. In the numerator
of \eqref{eq:forest-law}, each vertex contributes a factor $\pi(x)$. Dividing by
$\prod_{x\in V}\pi(x)$ leaves the weight
\[
  q^{|\R(F)|}\prod_{x\notin\R(F)}P\bigl(x,p_F(x)\bigr).
\]
Run Wilson's algorithm with an independent killing clock of rate $q$ for
each walk~\cite{Wilson1996,AvenaGaudilliere2018}. To recover
\eqref{eq:forest-law} from this exploration, it suffices to compute
the weight of one loop-erased branch. We give the calculation in
Section~\ref{sec:proof-root-identities}.

We denote the resulting forest by $F_q$ and its root set by $\Root$. For $x\in V$, we
write $\rho_q(x)$ for the root of the component containing $x$, and
$B_q(r):=\{x\in V:\rho_q(x)=r\}$
for the basin of the root $r$. We write $x\leftrightarrow_q y$ when $x$ and $y$ belong
to the same component of $F_q$.

Suppose we wish to know whether two distinct vertices $x_1,x_2$ are both roots.
Let us start Wilson's algorithm at $x_1$. If $x_1$ is a root, this first walk
was killed at $x_1$. It may have visited other vertices before it returned, but its
loop-erasure consists of $x_1$ alone. For $x_2$ to be a second root, the walk
started at $x_2$ must therefore be killed at $x_2$ before hitting $x_1$.
This is why a walk in $V\setminus\{x_1\}$ enters the calculation.

For a non-empty $D\subseteq V$, let $(S_t^D)_{t\geq0}$ be the ordinary evolving
set associated with the chain killed on leaving $D$. We give its construction in
Section~\ref{sec:construction}. The subscript $x$ on a probability or expectation
refers to the walk started from $x$. The subscript $\{x\}$ refers to the evolving
set started from that singleton. We omit the superscript $D$ when $D=V$.

\begin{theorem}[Root law and multi-root factorization]\label{thm:intro-roots}
Let $x,r\in V$. If $(S_t)_{t\geq0}$ is the ordinary evolving set in $V$ started from
$S_0=\{x\}$, then
\begin{equation}\label{eq:intro-root-law}
  \Pp\bigl(\rho_q(x)=r\bigr)
  =\frac{\pi(r)}{\pi(x)}
  \Pp_{\{x\}}\bigl(r\in S_{T_q}\bigr).
\end{equation}
In particular,
\begin{equation}\label{eq:intro-basin-law}
  \Ee\bigl[\pi(B_q(r))\bigr]=\pi(r).
\end{equation}
More generally, let $x_1,\ldots,x_k$ be distinct vertices of $V$ and set
\begin{equation*}
  A_0=\emptyset,\qquad A_i=\{x_1,\ldots,x_i\},\qquad
  D_{i-1}=V\setminus A_{i-1}\quad (1\leq i\leq k).
\end{equation*}
Then
\begin{equation}\label{eq:intro-multiroot}
  \Pp(x_1,\ldots,x_k\in\Root)
  =\prod_{i=1}^k
  \Pp^{D_{i-1}}_{\{x_i\}}
  \bigl(x_i\in S_{T_q}^{D_{i-1}}\bigr).
\end{equation}
\end{theorem}

The event on the left of Equality~\eqref{eq:intro-multiroot} is symmetric in the
prescribed vertices. On the right the order is that of the Wilson walks. Changing
this order can change the individual factors. We shall see in Section~\ref{sec:roots}
that it can also change the bound obtained by estimating the factors separately.

Avena and Gaudillière proved the determinantal formula for the roots with
kernel $K_q=q(qI+L)^{-1}$~\cite{AvenaGaudilliere2013,AvenaGaudilliere2018}.
In their proof~\cite[Proposition~2.2]{AvenaGaudilliere2018}, they already explore the
prescribed vertices in order. The invariance under permutations of the product of
successive diagonal Green functions is given in~\cite[Proposition~3.3]{Lawler2018}.
Equality~\eqref{eq:intro-multiroot} identifies each factor as a one-point inclusion
probability for the ordinary evolving set in the corresponding domain.

Let us now ask when two distinct vertices $x$ and $y$ belong to the same component.
Start Wilson's algorithm at $x$. The root of its component is the last vertex visited
before killing, so that
\[
  \Pp(\rho_q(x)=r)=\Pp_x(X_{T_q}=r)=K_q(x,r),\qquad r\in V.
\]
Applying this identity to both $x$ and $y$ gives
\[
  K_q(x,y)+K_q(y,x)
  \;\leq\; \Pp(x\leftrightarrow_q y)
  \;\leq\; \sum_{r\in V}\min\{K_q(x,r),K_q(y,r)\}.
\]
Indeed, if $y$ is the root of $x$, or if $x$ is the root of $y$, the two vertices
belong to the same component. These two events are disjoint since $x\ne y$.
For the upper bound, write $\{x\leftrightarrow_q y\}$ as the disjoint union over $r$
of the events $\{\rho_q(x)=\rho_q(y)=r\}$. For each $r$, the probability of this
event is at most the smaller of the two marginal probabilities.

To see how the two vertices become connected, let us keep the branch started at $x$
and run the second walk from $y$. We shall need the probability that this walk hits
a given set before it is killed. For $U\subset V$, put
$\tau_U:=\inf\{t\geq0:X_t\in U\}$ and
$h_q(x,U):=\Pp_x(\tau_U\leq T_q)$, with $h_q(x,U)=1$ for $x\in U$. For a singleton,
we write $h_q(x,y)$. Let $L_x$ be the loop-erasure of the first walk.

\begin{theorem}[Hitting before killing and forest connectivity]\label{thm:intro-hitting}
Let $U\subset V$ and $x\notin U$. Then
\begin{equation}\label{eq:intro-hitting-identity}
  h_q(x,U)
  =
  1-\frac{1}{\pi(x)}
  \Ee^{V\setminus U}_{\{x\}}
  \left[\pi\left(S_{T_q}^{V\setminus U}\right)\right].
\end{equation}
For $x,y\in V$, the branch and the component satisfy
\begin{equation}\label{eq:intro-branch-hitting}
  \Pp(y\in L_x)\leq h_q(x,y),
\end{equation}

\begin{equation}\label{eq:intro-connectivity}
  \Pp(x\leftrightarrow_q y)
  \leq
  \sum_{z\in V}h_q(x,z)h_q(y,z).
\end{equation}
\end{theorem}

Write $C_q(x)=\{y\in V:y\leftrightarrow_q x\}$. When $P$ is symmetric,
the sum over the intermediate vertex in Inequality~\eqref{eq:intro-connectivity}
can be estimated by the square of the resolvent. This gives a bound on the
mean size of $C_q(x)$ for every $q>0$. A heat-kernel estimate will then give
the dependence on the distance between $x$ and $y$.

\begin{theorem}[Mean component size and spatial decay]\label{thm:intro-gaussian}
Assume $\pi\equiv1$, and put $G_q=(qI+L)^{-1}$. For all $q>0$ and all $x,y\in V$,
\begin{equation}\label{eq:squared-resolvent-connectivity}
 \Pp(x\leftrightarrow_q y)
 \le (1+q)^2(G_q^2)(x,y)
 =(1+q)^2\int_0^\infty t e^{-qt}P_t(x,y)\dd t.
\end{equation}
In particular,
\begin{equation}\label{eq:intro-component-gaussian}
 \Ee[|C_q(x)|]\le (1+1/q)^2.
\end{equation}

Suppose now that $d_G$ is a graph distance, $n=|V|$, and that for some
$\nu>4$ and $C_H,c_H>0$ the heat kernel satisfies
\begin{equation}\label{eq:intro-gaussian-bound}
 P_t(x,y)\le C_H(1+t)^{-\nu/2}
     \exp\!\left(-c_H\frac{d_G(x,y)^2}{1+t}\right)+\frac{C_H}{n}
 \qquad (x,y\in V,\ t\ge0).
\end{equation}
Then, for all $0<q\le1$ and all $x,y\in V$, writing $r=d_G(x,y)$, we have
\begin{align}
 h_q(x,y)
 &\le C\left[(1+r)^{2-\nu}e^{-c\sqrt q\,r}+\frac{1}{qn}\right],
 \label{eq:intro-hitting-gaussian}\\
 \Pp(x\leftrightarrow_q y)
 &\le C\left[(1+r)^{4-\nu}e^{-c\sqrt q\,r}+\frac{1}{q^2n}\right].
 \label{eq:intro-connectivity-gaussian}
\end{align}
The constants $C,c>0$ depend only on $C_H,c_H$ and $\nu$.
\end{theorem}

The power $4-\nu$ in Inequality~\eqref{eq:intro-connectivity-gaussian} comes
from the extra factor $t$ in the integral in~\eqref{eq:squared-resolvent-connectivity}.
For distances beyond $q^{-1/2}$, the exponential factor gives further decay.
On $\mathbb Z^d$, $d>4$, the two-point function of the uniform spanning
forest has order $(1+|x-y|)^{4-d}$~\cite{BenjaminiKestenPeresSchramm2004}.
The same power appears here in an upper bound for the massive forest.
For the discrete tori in Section~\ref{sec:examples}, we obtain the spatial estimate
directly in the number of steps of the walk. At short times the constants in
\eqref{eq:intro-gaussian-bound} cannot be chosen independently of the side length.

Let us recall some related results. For uniform spanning forests and Wilson's
algorithm we refer to~\cite{Pemantle1991,Wilson1996,BLPS2001,LyonsPeres2016}.
Burton and Pemantle~\cite{BurtonPemantle1993} obtained determinantal formulas for
edges. For roots, the killed-walk kernel also enters the determinantal formula
when a set of roots is prescribed~\cite[Theorem~3]{AvenaCastellGaudilliereMelot2018}.
In Remark~\ref{rem:bernoulli-dpp} we recover the Bernoulli decomposition of
the number of roots in a region~\cite{HKPV2006} by expanding its generating
function. Corollary~\ref{cor:root-concentration} gives an upper tail bound
in terms of the invariant measure and the spectral gap.

Avena, Gaudillière, Milanesi and Quattropani~\cite{AvenaGaudilliereMilanesiQuattropani2022}
study two-point correlations on the complete graph and on weighted complete graphs
with a community structure. Avena, Driessen and
Koperberg~\cite{AvenaDriessenKoperberg2024} study their dependence on $q$ and their
behaviour in one-dimensional and tree-like geometries. Exact distributions and a local
limit for the massive forest on the complete graph are due to D'Achille, Enriquez and
Melotti~\cite{DAchilleEnriquezMelotti2026}. We use that graph to check the product
in~\eqref{eq:intro-multiroot} explicitly.

Morris and Peres introduced evolving sets in~\cite{MorrisPeres2005}. In the
conservative case the coupling between the walk and the evolving set biased by its
stationary mass is described in~\cite[Section~2.5]{AndersenPeres2009} using the
intertwining of Diaconis and Fill~\cite{DiaconisFill1990}. Here the walk can leave
$D$ before the exponential clock rings. In Section~\ref{sec:construction} we
keep track of the vertex hit on leaving $D$, since this determines where a Wilson
branch attaches to the forest already constructed. For the relation between
heat-kernel estimates, volume growth and Poincaré inequalities on graphs, we refer
to~\cite{Delmotte1999}.

\begin{remark}[The infinite massive Kirchhoff forest]\label{rem:infinite-massive-forest}
The massive forest also has a direct Wilson construction when $V$ is countable. Fix an
enumeration of $V$, add a cemetery state, and kill each Wilson walk at an independent
exponential time of parameter $q$. Since the chain makes only finitely many jumps on every
bounded time interval, each killed trajectory and its loop-erasure are finite almost surely.
Starting from the cemetery state and applying Wilson's algorithm successively to the
enumerated vertices therefore defines a spanning rooted forest, with exactly one root in
each component.

To check independence of the enumeration, fix a cylinder event involving the
outgoing edges of finitely many vertices. In a fixed enumeration, only finitely
many walks are needed before all these edges have been determined. Couple these
walks with their restrictions to a finite volume, killing also on exit from that
volume. Almost surely, the union of these finitely many trajectories run up to
their exponential clocks lies in a finite set. Along an exhaustion, the two
constructions therefore agree on the cylinder event for all sufficiently large
volumes. Since the finite-volume laws are independent of the enumeration, so is
the limit. The same exhaustion yields the finite-vertex Wilson identities, and
the corresponding bounds when passage to the limit in their right-hand sides
is justified.
\end{remark}

Sections~\ref{sec:roots} and~\ref{sec:hitting} develop the estimates for roots
and for components, respectively. Section~\ref{sec:examples} applies them to
particular graphs. We construct the coupling in Section~\ref{sec:construction}
and give the proofs in Section~\ref{sec:proofs}. Section~\ref{sec:resolvent}
relates inclusion probabilities and expected occupation times to the resolvent.

\newpage
\section{Consequences of the root identities}\label{sec:roots}

Suppose we wish to bound the probability that $x_1,\ldots,x_k$ are all roots.
In Equality~\eqref{eq:intro-multiroot}, each factor concerns a single vertex in the
domain left by the preceding roots. We may therefore estimate the factors separately.

\begin{corollary}[Abstract transfer to roots]\label{cor:abstract-transfer}
Suppose that for every non-empty $D\subseteq V$, every $x\in D$ and every $t\geq0$,
\begin{equation}\label{eq:abstract-transfer-assumption}
  \Pp^D_{\{x\}}(x\in S_t^D)\leq G_D(t,x).    
\end{equation}
Then, for distinct $x_1,\ldots,x_k\in V$ and $D_{i-1}=V\setminus\{x_1,\ldots,x_{i-1}\}$,
\begin{equation}\label{eq:abstract-transfer-bound}
  \Pp(x_1,\ldots,x_k\in\Root)
  \leq
  \prod_{i=1}^k\int_0^\infty qe^{-qt}G_{D_{i-1}}(t,x_i)\dd t.
\end{equation}
\end{corollary}

\subsection{Comparison with the determinantal formula}

Avena and Gaudillière~\cite{AvenaGaudilliere2013,AvenaGaudilliere2018} showed that the roots form a determinantal process with kernel
\begin{equation}\label{eq:root-kernel-definition}
  K_q:=q(qI+L)^{-1}.
\end{equation}
To compare its principal minors with
Equality~\eqref{eq:intro-multiroot}, fix a non-empty $D\subseteq V$. The Dirichlet
restriction $L_D$ acts on functions $f:D\to\mathbb R$ by
\begin{equation}\label{eq:dirichlet-generator-definition}
  L_Df(x)
  =
  f(x)-\sum_{y\in D}P(x,y)f(y),
  \qquad x\in D.
\end{equation}

\begin{proposition}[Probabilistic Schur-complement factorization]\label{prop:determinantal-comparison}
Let $x_1,\ldots,x_k$ be distinct vertices, $k\geq1$, and put $A=\{x_1,\ldots,x_k\}$.
In the exploration order $x_1,\ldots,x_k$, set $D_{i-1}=V\setminus\{x_1,\ldots,x_{i-1}\}$.
Then
\begin{equation}\label{eq:determinantal-comparison}
  \Pp(A\subset\Root)
  =\det K_q[A]
  =\prod_{i=1}^k q(qI+L_{D_{i-1}})^{-1}(x_i,x_i).
\end{equation}
The $i$th diagonal factor is the Schur-complement pivot obtained after eliminating
$x_1,\ldots,x_{i-1}$ from the principal submatrix $K_q[A]$.
\end{proposition}

\begin{remark}\label{rem:what-is-new}
The proof of the determinantal formula in~\cite[Proposition~2.2]{AvenaGaudilliere2018}
also starts Wilson's algorithm at the prescribed vertices. By
Equality~\eqref{eq:intro-multiroot}, the pivot for $x_i$ is the probability that
the ordinary evolving set in $D_{i-1}$, started from $\{x_i\}$, contains $x_i$ at
time $T_q$. It is this probability that we estimate below.
\end{remark}

Let us see how the factor at $y$ changes when we remove one vertex $x\ne y$
from $D$. Write $G_{q,D}=(qI+L_D)^{-1}$, $r_D(y)=qG_{q,D}(y,y)$ and
\[
  h_{q,D}(u,v)=\Pp_u(\tau_v<T_q\wedge\tau_{D^c}),\qquad u,v\in D.
\]
The entry $G_{q,D}(y,y)$ is the expected time spent at $y$ before killing or
exit. The part accumulated before the first visit to $x$ is
$G_{q,D\setminus\{x\}}(y,y)$. By the strong Markov property at $\tau_x$ and
the memoryless property of $T_q$, the remaining part is
$h_{q,D}(y,x)G_{q,D}(x,y)$. Applying the same argument at $\tau_y$ gives
$G_{q,D}(x,y)=h_{q,D}(x,y)G_{q,D}(y,y)$. Hence
\[
 r_{D\setminus\{x\}}(y)
 =r_D(y)\bigl(1-h_{q,D}(y,x)h_{q,D}(x,y)\bigr).
\]
The product on the right is the probability that the walk from $y$ visits $x$
and returns to $y$ before killing or exit.

\subsection{Dirichlet transfer in punctured domains}

For $f:V\to\mathbb R$, let $\cE(f,f)=\frac12\sum_{x,y\in V}\pi(x)P(x,y)(f(x)-f(y))^2$.
For each non-empty $D\subseteq V$, the first Dirichlet eigenvalue is
\begin{equation*}
  \lambda_1(D)
  :=\inf_{\substack{f\not\equiv0\\f=0\text{ on }D^c}}
  \frac{\cE(f,f)}{\norm{f}_{L^2(\pi)}^2}.
\end{equation*}
Write $Q(A,B)=\sum_{a\in A,b\in B}\pi(a)P(a,b)$. For $A\subseteq D$ the outward
flow splits as
\begin{equation*}
  Q_D^\partial(A)
  :=Q(A,V\setminus A)
  =Q(A,D\setminus A)+Q(A,D^c).
\end{equation*}
The second term counts jumps from $A$ that kill the walk on $D$. Both terms enter
the Dirichlet conductance profile
\begin{equation*}
  \Phi_D(r)
  :=\inf_{\substack{A\subseteq D\\0<\pi(A)\leq r}}
  \frac{Q(A,V\setminus A)}{\pi(A)}.
\end{equation*}

\begin{corollary}[Dirichlet spectral and Cheeger bounds]\label{cor:dirichlet-transfer}
Let $x_1,\ldots,x_k$ be distinct vertices of $V$. Then
\begin{equation}\label{eq:dirichlet-root-bound}
  \Pp(x_1,\ldots,x_k\in\Root)
  \leq
  \prod_{i=1}^k\frac{q}{q+\lambda_1(D_{i-1})}.
\end{equation}
With $\phi_D=\Phi_D(\pi(D))$, the Dirichlet Cheeger
inequality~\cite{LawlerSokal1988} gives
\begin{equation}\label{eq:cheeger-root-bound}
  \Pp(x_1,\ldots,x_k\in\Root)
  \leq
  \prod_{i=1}^k
  \frac{q}{q+\phi_{D_{i-1}}^2/2}.
\end{equation}
\end{corollary}

\subsection{A determinant-free spectral domination of root events}

On a one-vertex graph the sole vertex is a root with probability one. For the
estimates involving the spectral gap, assume $|V|\geq2$ and let $\gamma$ be
the smallest positive eigenvalue of $L$ on the whole graph.
Let $\bar\pi=\pi/\pi(V)$ and set
\begin{equation}\label{eq:beta-definition}
  \forall x\in V\,,\qquad
  \beta_q(x)
  :=
  \bar\pi(x)
  +\bigl(1-\bar\pi(x)\bigr)\frac{q}{q+\gamma}.
\end{equation}

\begin{theorem}[Spectral domination of root events]\label{thm:spectral-domination}
Let $A=\{x_1,\ldots,x_k\}$, where the vertices $x_i$ are distinct. Then
\begin{equation}\label{eq:combined-root-product}
  \Pp(A\subset\Root)
  \leq
  \prod_{i=1}^k
  \min\left\{
    \beta_q(x_i),
    \frac{q}{q+\lambda_1(D_{i-1})}
  \right\}.
\end{equation}
In particular,
\begin{equation}\label{eq:root-product-domination}
  \Pp(A\subset\Root)
  \leq\prod_{x\in A}\beta_q(x).
\end{equation}
\end{theorem}

\paragraph{The order of exploration on a path.}
Let us compare the two orders in Theorem~\ref{thm:spectral-domination} on the path
$1$--$2$--$3$. Take $q=1$ and let $P$ be the simple random walk. We wish to estimate the probability that vertices $1$ and $2$ are both roots. To compute the factors in Equality~\eqref{eq:determinantal-comparison},
we use
\[
  2I-P=
  \begin{pmatrix}
      2&-1&0\\[-1mm]
      -1/2&2&-1/2\\[-1mm]
      0&-1&2
  \end{pmatrix}.
\]
The first two diagonal entries of the inverse are $7/12$ and $2/3$.
If we start at $1$, we next delete row and column $1$ and take the diagonal entry
at $2$ of the inverse of the remaining matrix. This gives $4/7$. If we start at $2$,
deleting row and column $2$ leaves $2I$, so the second factor is $1/2$.
Thus the two orders give, respectively,
\[
   \frac{7}{12}\,\frac47=\frac13,
   \qquad
   \frac23\,\frac12=\frac13.
\]

The walk from $1$ reaches $2$ before killing with probability $1/2$.
From $2$ the probability $u$ of reaching $1$ satisfies $u=1/4+u/8$, by
conditioning on the first step and, if it leads to $3$, on the return to $2$.
The probability of going from $1$ to $2$ and back before killing is therefore
$(1/2)(2/7)=1/7$. Thus removing $1$ multiplies $2/3$ by $6/7$, and removing
$2$ multiplies $7/12$ by the same factor.

Let us now apply Inequality~\eqref{eq:combined-root-product} to the same event.
Here $\bar\pi=(1/4,1/2,1/4)$ and $I-P$ has eigenvalues $0,1,2$.
It follows that $\beta_1(1)=5/8$ and $\beta_1(2)=3/4$.
In either order the first domain is $V$ and $\lambda_1(V)=0$.
After removing vertex $1$ we have $\lambda_1(\{2,3\})=1-1/\sqrt2$,
whereas after removing vertex $2$ we have $\lambda_1(\{1,3\})=1$. The two
upper bounds are therefore
\[
 \begin{aligned}
 (1,2):&\quad \frac58\min\!\left\{\frac34,\frac{1}{2-1/\sqrt2}\right\}
       =\frac{15}{32},\\
 (2,1):&\quad \frac34\min\!\left\{\frac58,\frac12\right\}
       =\frac38.
 \end{aligned}
\]
The exact probability is $1/3$ in both orders. Removing the middle vertex first
leaves two isolated vertices. It is at this second step that the Dirichlet
estimate improves the bound.

\subsection{Root moments and a determinant-free concentration corollary}

Write $N_q=|\Root|$ and $N_q(U)=|\Root\cap U|$ for $U\subseteq V$. The falling
factorial $(N_q)_k=N_q(N_q-1)\cdots(N_q-k+1)$ counts ordered $k$-tuples of distinct roots.
\begin{proposition}[Root moments, void events and basins]\label{prop:root-moments}
For each integer $k\geq1$,
\begin{equation*}
  \Ee[(N_q)_k]
  =\sum_{\substack{x_1,\ldots,x_k\in V\\\mathrm{distinct}}}
  \prod_{i=1}^k
  \Pp^{V\setminus A_{i-1}}_{\{x_i\}}
  \bigl(x_i\in S_{T_q}^{V\setminus A_{i-1}}\bigr).
\end{equation*}
Restricting the sum to vertices in $U$ gives
\begin{equation*}
  \Ee[(N_q(U))_k]
  =\sum_{\substack{x_1,\ldots,x_k\in U\\\mathrm{distinct}}}
  \prod_{i=1}^k
  \Pp^{V\setminus A_{i-1}}_{\{x_i\}}
  \bigl(x_i\in S_{T_q}^{V\setminus A_{i-1}}\bigr).
\end{equation*}
The probability that $U$ contains no roots is
\begin{equation}\label{eq:void-event}
  \Pp(\Root\cap U=\emptyset)
  =\sum_{A\subset U}(-1)^{|A|}\Pp(A\subset\Root).
\end{equation}
For each $x\in V$,
\begin{equation}\label{eq:root-in-set}
  \Pp(\rho_q(x)\in U)
  =\frac{1}{\pi(x)}\Ee_{\{x\}}[\pi(S_{T_q}\cap U)].
\end{equation}
Thus the union $B_q(U)=\{x\in V:\rho_q(x)\in U\}$ of the corresponding basins satisfies
\begin{equation}\label{eq:collective-basin}
  \Ee[\pi(B_q(U))]=\pi(U).
\end{equation}
Since each component has one root, the expected number of edges is
\begin{equation*}
  \Ee[|E(F_q)|]=|V|-\Ee[N_q].
\end{equation*}
\end{proposition}

\begin{corollary}[Poisson-type concentration of roots]\label{cor:root-concentration}
Taking $k=1$ in Theorem~\ref{thm:spectral-domination} gives
$\Ee[N_q(U)]\leq\mu_q(U):=\sum_{x\in U}\beta_q(x)$. For any non-negative numbers
$(s_x)_{x\in U}$,
\begin{equation*}
  \Ee\left[
    \prod_{x\in U}\left(1+s_x\1_{\{x\in\Root\}}\right)
  \right]
  \leq
  \prod_{x\in U}(1+s_x\beta_q(x)).
\end{equation*}
Taking $s_x=e^\theta-1$, $\theta\geq0$, gives
\begin{equation}\label{eq:root-mgf}
  \Ee[e^{\theta N_q(U)}]
  \leq
  \prod_{x\in U}\left(1+\beta_q(x)(e^\theta-1)\right)
  \leq
  \exp\left(\mu_q(U)(e^\theta-1)\right).
\end{equation}
For every $\delta>0$,
\begin{equation}\label{eq:root-chernoff}
  \Pp\left(N_q(U)\geq(1+\delta)\mu_q(U)\right)
  \leq
  \exp\left[-\mu_q(U)
  \bigl((1+\delta)\log(1+\delta)-\delta\bigr)\right].
\end{equation}
For each integer $k\geq1$, with $e_k$ the $k$th elementary symmetric polynomial
(taken to be zero if $k>|U|$),
\begin{equation*}
  \Ee[(N_q(U))_k]
  \leq
  k!\,e_k\bigl((\beta_q(x))_{x\in U}\bigr)
  \leq\mu_q(U)^k.
\end{equation*}
\end{corollary}

\begin{remark}[Comparison with the Bernoulli decomposition]\label{rem:bernoulli-dpp}
Let us recover the distribution of $N_q(U)$ from the inclusion probabilities.
This is the Bernoulli decomposition for a finite determinantal
process~\cite{HKPV2006}. Since $L$ is non-negative and self-adjoint on
$L^2(V,\pi)$, each eigenvalue of $K_q$ has the form $q/(q+\lambda)$
with $\lambda\geq0$. For a function on $U$, extend it by zero to $V$.
Its quadratic form under $K_q[U]$ is then its quadratic form under $K_q$.
Thus $K_q[U]$ is a positive self-adjoint contraction in $L^2(U,\pi)$.
In particular its eigenvalues $\kappa_1,\ldots,\kappa_{|U|}$ lie in $[0,1]$.

For $z>0$, expanding the product over $U$ and using
Equality~\eqref{eq:determinantal-comparison} gives
\begin{align*}
 \Ee[z^{N_q(U)}]
 &=\Ee\!\left[\prod_{x\in U}(1+(z-1)\1_{\{x\in\Root\}})\right]\\
 &=\sum_{A\subseteq U}(z-1)^{|A|}\det K_q[A]
  =\det\bigl(I+(z-1)K_q[U]\bigr).
\end{align*}
For the last equality, expand the determinant column by column. Choosing
the columns of $(z-1)K_q[U]$ indexed by $A$ and the columns of the identity
elsewhere gives precisely the principal minor indexed by $A$, with
$\det K_q[\varnothing]=1$. The determinant is therefore
$\prod_{j=1}^{|U|}(1-\kappa_j+\kappa_j z)$. These are the factors of the
generating function of a sum of independent Bernoulli variables. Comparing
coefficients gives
\[
 N_q(U)\overset{d}=\sum_{j=1}^{|U|}\xi_j,
 \qquad \Pp(\xi_j=1)=\kappa_j,
\]
where the $\xi_j$ are independent. Setting $z=e^\theta$ gives the exact
moment generating function. Corollary~\ref{cor:root-concentration}
uses only $\bar\pi$ and the spectral gap of $L$ to give an upper bound.
\end{remark}

\pagebreak[2]
\section{Consequences of the hitting identity}\label{sec:hitting}

We keep $x$ fixed and ask how far its branch can reach, and how many vertices can
belong to its component. Branch estimates use the range of a single killed walk.
For the component we use Inequality~\eqref{eq:intro-connectivity}, which involves
the walks started from both vertices.

\subsection{Gaussian localization of components}

\begin{remark}[Finite-volume correction]\label{rem:finite-volume-floor}
On a finite graph with $\pi\equiv1$, $P_t(x,y)$ tends to $1/n$ as
$t\to\infty$. Integrating the constant term $C_H/n$ from
\eqref{eq:intro-gaussian-bound} against $t e^{-qt}$ gives the correction
$C_H/(q^2n)$ in the integral for connectivity. For the infinite forest of
Remark~\ref{rem:infinite-massive-forest}, suppose instead that $\pi\equiv1$ and
\[
 P_t(x,y)\le C_H(1+t)^{-\nu/2}
 \exp\!\left(-c_H\frac{d_G(x,y)^2}{1+t}\right)
 \qquad(x,y\in V,\ t\ge0)
\]
for some $\nu>4$. The same integral then yields
$\Pp(x\leftrightarrow_q y)\le C(1+d_G(x,y))^{4-\nu}
e^{-c\sqrt q\,d_G(x,y)}$ for $0<q\le1$.
\end{remark}

\subsection{Component sizes and susceptibility}

\begin{proposition}[Component sizes and susceptibility]\label{prop:components}
For every $x\in V$,
\begin{equation}\label{eq:mean-component}
  \Ee[\pi(C_q(x))]
  \leq
  \sum_{z\in V}h_q(x,z)
  \sum_{y\in V}\pi(y)h_q(y,z).
\end{equation}
Weighting by $\pi(x)$ and summing gives a bound for the susceptibility
$\chi_q=\sum_{x,y\in V}\pi(x)\pi(y)\Pp(x\leftrightarrow_q y)$,
\begin{equation}\label{eq:susceptibility}
  \chi_q
  \leq
  \sum_{z\in V}
  \left(\sum_{x\in V}\pi(x)h_q(x,z)\right)^2.
\end{equation}
If $M_q$ is the largest component mass, then for each $a>0$,
\begin{equation}\label{eq:largest-component}
  \Pp(M_q\geq a)
  \leq\frac{\chi_q}{a^2}
  \leq
  \frac{1}{a^2}
  \sum_{z\in V}
  \left(\sum_{x\in V}\pi(x)h_q(x,z)\right)^2.
\end{equation}
For a prescribed set $U\subseteq V$ we also have
\begin{equation}\label{eq:component-hits-set}
  \Pp(C_q(x)\cap U\neq\emptyset)
  \leq
  \sum_{y\in U}\sum_{z\in V}h_q(x,z)h_q(y,z).
\end{equation}
\end{proposition}

\subsection{Spatial consequences}

Fix a graph distance $d$ and write $B(x,R)=\{y\in V:d(x,y)\leq R\}$. Let
$R_x^{\mathrm{raw}}=\{z\in V:\tau_z\leq T_q\}$ be the range of the killed walk
from $x$, before its loops are erased.

\begin{proposition}[Range, radius, depth and diameter]\label{prop:spatial}
For every $x\in V$,
\begin{equation}\label{eq:raw-range}
  \Ee[\pi(R_x^{\mathrm{raw}})]
  =\sum_{z\in V}\pi(z)h_q(x,z),
\end{equation}
and $L_x\subseteq R_x^{\mathrm{raw}}$ gives
\begin{equation}\label{eq:branch-size}
  \Ee[\pi(L_x)]
  \leq\sum_{z\in V}\pi(z)h_q(x,z).
\end{equation}
For $\operatorname{rad}(L_x)=\max_{y\in L_x}d(x,y)$ and $R\geq0$,
\begin{equation}\label{eq:radius}
  \Pp(\operatorname{rad}(L_x)>R)
  \leq h_q(x,B(x,R)^c)
\end{equation}
and the hitting identity in Theorem~\ref{thm:intro-hitting} gives
\begin{equation}\label{eq:radius-es}
  \Pp(\operatorname{rad}(L_x)>R)
  \leq
  1-\frac{1}{\pi(x)}
  \Ee^{B(x,R)}_{\{x\}}\left[\pi(S_{T_q}^{B(x,R)})\right].
\end{equation}
Let $\operatorname{depth}_q(x)$ be the number of edges from $x$ to $\rho_q(x)$ in the
forest. For every integer $m\geq0$,
\begin{equation}\label{eq:depth-tail}
  \Pp(\operatorname{depth}_q(x)\geq m)
  \leq\left(\frac{1}{1+q}\right)^m,
\end{equation}
and summation over $m\geq1$ gives
\begin{equation}\label{eq:mean-depth}
  \Ee[\operatorname{depth}_q(x)]\leq\frac1q.
\end{equation}
With the diameter measured in $d$, for every $R\geq0$ we have
\begin{equation}\label{eq:local-diameter}
  \Pp(\operatorname{diam}(C_q(x))\geq R)
  \leq
  \sum_{\substack{y\in V\\d(x,y)\geq R/2}}
  \sum_{z\in V}h_q(x,z)h_q(y,z),
\end{equation}
and, for the whole forest,
\begin{equation}\label{eq:global-diameter}
  \Pp\bigl(\exists C:\operatorname{diam}(C)\geq R\bigr)
  \leq
  \sum_{\substack{x,y\in V\\d(x,y)\geq R}}
  \sum_{z\in V}h_q(x,z)h_q(y,z).
\end{equation}
\end{proposition}

\subsection{Upper escape profile}

How large can the flow into $U$ be from a set of a given mass? For $D_U=V\setminus U$
and $0\leq r\leq\pi(D_U)$, put
\begin{equation}\label{eq:escape-profile-definition}
  \Psi_U(r)
  :=
  \sup\{Q(A,U):A\subset D_U,\ \pi(A)\leq r\}.
\end{equation}
Let $\widehat\Psi_U$ be the smallest increasing concave majorant of $\Psi_U$.

\begin{theorem}[Escape-profile bound]\label{thm:escape-profile}
For $x\in D_U$, let $u_{x,U}$ solve
\begin{equation}\label{eq:escape-comparison-ode}
  u_{x,U}'(t)
  =
  -\widehat\Psi_U(u_{x,U}(t)),
  \qquad
  u_{x,U}(0)=\pi(x).
\end{equation}
Then
\begin{equation}\label{eq:escape-bound}
  h_q(x,U)
  \leq
  H_q(x,U)
  :=1-\frac{q}{\pi(x)}\int_0^\infty e^{-qt}u_{x,U}(t)\dd t.
\end{equation}
For $x\in U$, set $H_q(x,U)=1$.
\end{theorem}

\begin{corollary}[Linear escape bound]\label{cor:linear-escape}
For $U\ne V$ let
\begin{equation}\label{eq:linear-escape-rate}
  \lambda_U
  :=
  \sup_{a\in V\setminus U}P(a,U).
\end{equation}
For every $x\notin U$,
\begin{equation}\label{eq:linear-escape-bound}
  h_q(x,U)
  \leq
  \frac{\lambda_U}{q+\lambda_U}.
\end{equation}
\end{corollary}

\begin{corollary}[Use of deterministic upper bounds]\label{cor:replace-h}
Suppose $H_q(x,U)\geq h_q(x,U)$ for every vertex $x$ and every set $U$, and write
$H_q(x,z)=H_q(x,\{z\})$. If the right-hand side of an identity or an upper
bound is a sum of products of hitting probabilities with non-negative coefficients,
replacing $h_q$ by $H_q$ gives an upper bound. In particular,
\begin{equation*}
  \forall x,y \in V\,,\qquad \Pp(x\leftrightarrow_q y)
  \leq\sum_{z\in V}H_q(x,z)H_q(y,z),
\end{equation*}
and
\begin{equation*}
    \forall x \in V\,,\qquad  \Ee[\pi(R_x^{\mathrm{raw}})]
  \leq\sum_{z\in V}\pi(z)H_q(x,z).
\end{equation*}
\end{corollary}
\pagebreak[2]
\section{Examples}\label{sec:examples}

\subsection{The complete graph}

Let us compute the factors in Equality~\eqref{eq:intro-multiroot} for simple random
walk on the complete graph. Exact distributional results for the forest on this graph
are given in~\cite{DAchilleEnriquezMelotti2026}. Here $V=\{1,\ldots,n\}$ with
$n\geq2$, $\pi\equiv1$, and $P(x,y)=\1_{\{y\ne x\}}/(n-1)$.
Until the walk hits a given vertex $y$, its rate of jumping to $y$ is $1/(n-1)$.
Consequently, for $x\ne y$,
\[
  h_q(x,y)=\frac{1}{1+(n-1)q}=: \alpha_n(q),
  \qquad h_q(y,y)=1.
\]
In Inequality~\eqref{eq:intro-connectivity}, the terms $z=x$ and $z=y$ each
contribute $\alpha_n(q)$. The remaining $n-2$ terms give
\[
  \Pp(x\leftrightarrow_q y)\leq 2\alpha_n(q)+(n-2)\alpha_n(q)^2.
\]
Summing Inequality~\eqref{eq:intro-connectivity} over all $y$ gives
\[
  \Ee[|C_q(x)|]\leq\bigl(1+(n-1)\alpha_n(q)\bigr)^2
                  \leq(1+1/q)^2.
\]

Consider now a domain $D$ with $|D|=m\geq1$. On $D$, the constant functions are
eigenfunctions of $L_D$ with eigenvalue $(n-m)/(n-1)$. On their orthogonal complement
the eigenvalue is $n/(n-1)$. For $x\in D$, this gives
\[
  \Pp_x(X_t=x,\ t<\tau_{D^c})
     =\frac1m e^{-\frac{n-m}{n-1}t}
       +\left(1-\frac1m\right)e^{-\frac{n}{n-1}t}.
\]
Averaging at $T_q$ and putting $a=(n-1)q$, we obtain
\[
  r_m(q):=\Pp^D_{\{x\}}(x\in S^D_{T_q})
  =\frac{1}{m}\frac{a}{a+n-m}
      +\left(1-\frac1m\right)\frac{a}{a+n}
  =\frac{a(a+n-m+1)}{(a+n-m)(a+n)}.
\]
At each step of the exploration the domain loses one vertex. For $|A|=k\geq1$,
the resulting product telescopes:
\[
 \Pp(A\subset\Root)=\prod_{j=0}^{k-1}r_{n-j}(q)
       =\frac{a^{k-1}(a+k)}{(a+n)^k}.
\]
In particular, $\Pp(x\in\Root)=(a+1)/(a+n)$ and, for distinct $x,y$,
$\Pp(x,y\in\Root)=a(a+2)/(a+n)^2$.
The root kernel on $V$ has eigenvalue $1$ on the constants and $a/(a+n)$ on
their $(n-1)$-dimensional orthogonal complement. Applying
Remark~\ref{rem:bernoulli-dpp} with $U=V$, we obtain
$N_q\overset{d}=1+\operatorname{Bin}(n-1,a/(a+n))$.

\subsection{Discrete tori}

Let $P$ be the simple random walk kernel on $V=(\mathbb Z/N\mathbb Z)^d$,
$N\geq2$, with $\pi\equiv1$. At each step we choose uniformly among the $2d$
coordinate directions.
Write $r=d_N(x,y)$ for the graph distance. At least $r$ jumps are needed to reach $y$
from $x$, so that $h_q(x,y)\leq(1+q)^{-r}$. For $0<q\leq1$,
summing this estimate over $y$ gives at most $C_dq^{-d}$, by summing a geometric
series in each coordinate. Used for both hitting probabilities in
\eqref{eq:intro-connectivity}, it gives only $C_dq^{-2d}$ for the expected component size.

Write $G_q=(qI+I-P)^{-1}$. We shall use the two inequalities
\[
 h_q(x,y)\leq(1+q)G_q(x,y),\qquad
 \Pp(x\leftrightarrow_q y)\leq(1+q)^2G_q^2(x,y)
\]
from \eqref{eq:resolvent-hitting-ratio}, \eqref{eq:diagonal-resolvent-lower}
and Theorem~\ref{thm:intro-gaussian}.
The second bound uses only the symmetry of $P$, and summing over $y$ gives
$\Ee[|C_q(x)|]\leq(1+q)^2/q^2$ in every dimension.

To estimate the probability of connection at distance $r$, we work in discrete
time. This avoids a short-time difficulty with Inequality~\eqref{eq:intro-gaussian-bound}:
at time one the continuous-time walk can make exactly $r$ steps, all in the first
coordinate direction, with probability $e^{-1}(2d)^{-r}/r!$. Let $r$ tend to infinity
and take $N=\lceil e^{r^2}\rceil$. The endpoint then has distance $r$, and this
probability eventually exceeds $Ce^{-cr^2}+C/N^d$ for any fixed $C,c>0$.

Put $H=(I+P)/2$ and $s=q/2$. Since $I-P=2(I-H)$,
$G_q=\tfrac12((1+s)I-H)^{-1}$. Expanding the resolvent and its square gives,
for $j=1,2$,
\[
 G_q^j(x,y)=\frac{1}{2^j(1+s)^j}
 \sum_{k\geq0}\binom{k+j-1}{j-1}(1+s)^{-k} H^k(x,y).
\]
Here $H^k(x,y)=0$ when $k<r$. We next bound the summands uniformly in $N$.
Let us first consider the lazy walk $Z$ on $\mathbb Z^d$ started at $0$ and put
$p_k(z)=\Pp(Z_k=z)$. Its one-step Fourier transform on $[-\pi,\pi]^d$ is
\[
 \phi(\xi)=\frac12+\frac1{2d}\sum_{i=1}^d\cos\xi_i.
\]
The holding probability makes $\phi$ non-negative. Since
$1-\cos u\geq 2u^2/\pi^2$ for $|u|\leq\pi$, we have
$\phi(\xi)\leq\exp(-|\xi|_2^2/(d\pi^2))$. Fourier inversion now gives, for $k\geq1$,
\[
 \sup_z p_k(z)\leq (2\pi)^{-d}\int_{[-\pi,\pi]^d}\phi(\xi)^k\dd\xi
                 \leq C_d k^{-d/2}.
\]

To recover the dependence on $z$, consider one coordinate of a single step.
Its exponential moment at $u\in\mathbb R$ is
$1+(\cosh u-1)/(2d)\leq\cosh u\leq e^{u^2/2}$, the last inequality following
by comparison of the power series. If $|Z_k|_2\geq R$, some signed coordinate is
at least $R/\sqrt d$. By independence of the steps and the exponential Markov
inequality with $u=R/(k\sqrt d)$,
\[
 \Pp(|Z_k|_2\geq R)\leq 2d\exp\bigl(-R^2/(2dk)\bigr),\qquad k\geq1.
\]
For $k\geq2$, write $a=\lfloor k/2\rfloor$, $b=k-a$ and split the walk at step $a$.
In the convolution $p_k(z)=\sum_v p_a(v)p_b(z-v)$, either $|v|_2\geq |z|_2/2$
or $|z-v|_2\geq |z|_2/2$. It follows that
\[
 p_k(z)\leq \|p_b\|_\infty\Pp(|Z_a|_2\geq |z|_2/2)
            +\|p_a\|_\infty\Pp(|Z_b|_2\geq |z|_2/2).
\]
Both $a$ and $b$ are at least $k/3$. Combining the preceding two estimates gives
$p_k(z)\leq C_d(1+k)^{-d/2}e^{-c_d |z|_2^2/(1+k)}$.
The same bound holds for $k=0,1$ after increasing $C_d$, since then $|Z_k|_2\leq1$.
For Gaussian estimates for uniformly elliptic lazy reversible walks on graphs
with volume doubling and a scale-invariant Poincar\'e inequality,
see~\cite{Delmotte1999}.

Choose a representative $w$ of $y-x$ with coordinates in $[-N/2,N/2]$.
The torus kernel is the sum of the lattice kernels from $0$ to $w+N\ell$ over
$\ell\in\mathbb Z^d$. Each such displacement has $\ell^1$-norm at least $r$ and
Euclidean norm at least $N|\ell|_2/2$. Writing $m=1+k$ and splitting the Gaussian
exponent between these two bounds, we get
\[
 H^k(x,y)\leq C_d m^{-d/2}e^{-c_d r^2/m}
    \sum_{\ell\in\mathbb Z^d}e^{-c_dN^2|\ell|_2^2/m}
 \leq C_d m^{-d/2}e^{-c_d r^2/m}+\frac{C_d}{N^d}.
\]
In the last inequality, the sum is at most $C_d(1+m^{d/2}/N^d)$, as follows by
summing in each coordinate.

Suppose $0<q\leq1$. The factor $(1+s)^{-k}$ is at most $C e^{-cqm}$. The constant
term $C_d/N^d$ in the heat-kernel bound therefore contributes at most
$C_d q^{-j}/N^d$ to $G_q^j(x,y)$. For the other term, use
$q m+r^2/m\geq2\sqrt q\,r$, leaving half of the Gaussian exponent in the sum.
When $d>2j$,
\[
 \sum_{m\geq1}m^{j-1-d/2}e^{-c r^2/m}
       \leq C(1+r)^{2j-d}.
\]
Indeed, for $r\geq1$ sum first over $m\geq r^2$. On each remaining interval
$2^{-i-1}r^2<m\leq2^{-i}r^2$, $i\geq0$, the exponential contributes at most
$e^{-c2^i}$. For $r=0$ the series converges because $d>2j$.

Applying these estimates with $j=1$ to the hitting probability and with $j=2$ to
the connection probability, we conclude that for $d>4$,
\begin{align*}
 h_q(x,y)&\leq C_d\left[(1+r)^{2-d}e^{-c_d\sqrt q\,r}+\frac1{qN^d}\right],\\
 \Pp(x\leftrightarrow_q y)&\leq C_d\left[(1+r)^{4-d}e^{-c_d\sqrt q\,r}
                           +\frac1{q^2N^d}\right].
\end{align*}
The hitting estimate already holds for $d>2$. For connectivity, the extra factor
$k+1$ in the expansion of $G_q^2$ requires $d>4$ in the sum above.

\subsection{Two graphs connected by one edge}

Let $V=V_1\sqcup V_2$, with a single edge $\{a,b\}$ between the two connected
weighted graphs, $a\in V_1$ and $b\in V_2$. Write $c_{ab}=\pi(a)P(a,b)$ for
its conductance and take $U=V_2$. For $A\subset V_1$ the flow across the edge is
$Q(A,V_2)=c_{ab}\1_{\{a\in A\}}$. A set of mass less than $\pi(a)$ cannot
contain $a$, while the singleton $\{a\}$ already carries all the flow. Hence
\[
  \Psi_{V_2}(r)=
  \begin{cases}
    0,&0\leq r<\pi(a),\\
    c_{ab},&\pi(a)\leq r\leq\pi(V_1).
  \end{cases}
  \qquad
  \widehat\Psi_{V_2}(r)=\min\{rP(a,b),c_{ab}\}.
\]
The linear estimate in Corollary~\ref{cor:linear-escape} gives
$h_q(x,V_2)\leq P(a,b)/(q+P(a,b))$ for every $x\in V_1$.
If $\pi(x)\leq\pi(a)$, the solution of
Equality~\eqref{eq:escape-comparison-ode} starts in the linear part of the profile
and stays there. If $\pi(x)>\pi(a)$, it first decreases at the constant rate
$c_{ab}$ until its mass reaches $\pi(a)$. This is the case in which the full
profile improves the linear estimate.

At the endpoint $a$ we can also compute the crossing probability from a
single diagonal factor. Put $D=V_1$. At time $t$, the probability that the
walk from $x\in D$ is at $a$ without having crossed the edge is
$(e^{-tL_D})(x,a)$. From $a$ it crosses at rate $P(a,b)$. Integrating this
flow up to the independent killing time gives
\[
 h_q(x,V_2)=P(a,b)\int_0^\infty e^{-qt}(e^{-tL_D})(x,a)\dd t
           =P(a,b)G_{q,D}(x,a).
\]
In particular, by Equality~\eqref{eq:punctured-resolvent},
\[
 h_q(a,V_2)=\frac{P(a,b)}q\,r_D(a)
           =\frac{P(a,b)}q\,\Pp^D_{\{a\}}(a\in S^D_{T_q}).
\]
For $x\in V_1\setminus\{a\}$ the walk must first reach $a$, so that
$h_q(x,V_2)=h_q(x,a)h_q(a,V_2)$.

\subsection{The hypercube}

Let $P$ flip one uniformly chosen coordinate of $V=\{0,1\}^m$, $m\geq1$.
Write $d_H$ for the Hamming distance and $a_q=(1+q)^{-1}$. Counting the jumps
before killing gives $h_q(x,y)\leq a_q^{d_H(x,y)}$. Summing over the coordinates,
\[
  \sum_{z\in V}a_q^{d_H(x,z)}=(1+a_q)^m.
\]
With counting measure, it follows that
$\Ee[|C_q(x)|]\leq (1+a_q)^{2m}$ and $\chi_q\leq 2^m(1+a_q)^{2m}$.
For the two-point estimate, set $R=d_H(x,y)$. A coordinate in which $x$ and $y$
differ contributes $2a_q$ to the sum in Inequality~\eqref{eq:intro-connectivity}.
A coordinate on which they agree contributes $1+a_q^2$. Thus
\[
  \Pp(x\leftrightarrow_q y)\leq(2a_q)^R(1+a_q^2)^{m-R}.
\]
For $0<q\leq1$, the right-hand side is at least one, even when $x$ and $y$
are opposite corners. In this range the jump count alone gives no improvement over the
trivial two-point bound. For the mean size the symmetry argument above still gives
$(1+1/q)^2$, uniformly in $m$.

\subsection{A determinant-free root bound on expanders}

Consider a family of chains for which $\bar\pi$ is uniform and $\gamma\geq\gamma_0>0$.
Write $n=|V|$. By the definition of $\beta_q$,
\[
  \beta_q(x)=\frac{nq+\gamma}{n(q+\gamma)},\qquad
  \mu_q(V)=\frac{nq+\gamma}{q+\gamma}.
\]
Take $q=\vartheta/n$ with $\vartheta>0$ fixed. Then
$\mu_q(V)\leq M_0:=1+\vartheta/\gamma_0$.
Applying Inequality~\eqref{eq:root-mgf} gives
\[
  \Ee[e^{\theta N_{\vartheta/n}}]\leq
     \exp\!\left[M_0(e^\theta-1)\right],\qquad \theta\geq0 .
\]
In particular the root counts form a tight family. For each integer $m\geq M_0$,
taking $\theta=\log(m/M_0)$ in the exponential Markov inequality yields
\[
  \Pp(N_{\vartheta/n}\geq m)\leq(e M_0/m)^m.
\]
Only the lower bound on the spectral gap enters this estimate.

\pagebreak[2]
\section{The punctured Wilson--evolving-set coupling}\label{sec:construction}

\subsection{Ordinary evolving sets in a punctured domain}

Let $D\subseteq V$ be non-empty and let $\tau_{D^c}$ be the first exit time from $D$.
For $A\subseteq D$ and $y\in D$, put
$Q_D(A,y)=\sum_{a\in A}\pi(a)P(a,y)$. Let $(U_n)_{n\geq1}$ be independent uniform
random variables on $[0,1]$. Starting from $S_0^D=A$, define
\begin{equation}\label{eq:punctured-evolving-set-definition}
  S_{n+1}^D
  =\left\{y\in D:\frac{Q_D(S_n^D,y)}{\pi(y)}\geq U_{n+1}\right\},
  \qquad n\geq0.
\end{equation}
We obtain the continuous-time process $S_t^D$ by running this discrete chain on an
independent Poisson clock of rate one. Its one-point inclusion probabilities give
the transition probabilities of the walk killed on leaving $D$.

\begin{lemma}[Punctured intertwining]\label{lem:punctured-intertwining}
For $A\subseteq D$, $y\in D$ and $t\geq0$,
\begin{equation}\label{eq:punctured-intertwining}
  \Pp_A^D(y\in S_t^D)
  =\frac{1}{\pi(y)}
  \sum_{a\in A}\pi(a)\Pp_a(X_t=y,\ t<\tau_{D^c}).
\end{equation}
In particular, when $A=\{x\}$,
\begin{equation}\label{eq:singleton-intertwining}
  \Pp_{\{x\}}^D(y\in S_t^D)
  =\frac{\pi(x)}{\pi(y)}\Pp_x(X_t=y,\ t<\tau_{D^c}).
\end{equation}
\end{lemma}

\subsection{The transformed evolving set and the marked trajectory}

Let $K_D$ be the one-step transition kernel of the ordinary evolving set on subsets
of $D$. For non-empty $A,B\subseteq D$, put
\begin{equation}\label{eq:transformed-kernel-definition}
  \wh K_D(A,B)=K_D(A,B)\frac{\pi(B)}{\pi(A)}.
\end{equation}
Starting from $A$, the expected mass after one step is
$\sum_{B\subseteq D}K_D(A,B)\pi(B)=\pi(A)-Q(A,D^c)$. Thus $\wh K_D$ loses
mass $Q(A,D^c)/\pi(A)$ on that step. Send this mass to a cemetery state $\dagger$.
For $\varnothing\ne A\subseteq D$ and $y\in D$, define
\begin{equation}\label{eq:link-definition}
  \Lambda(A,y)=\frac{\pi(y)}{\pi(A)}\1_{\{y\in A\}} .
\end{equation}
Write $P_D$ for the substochastic restriction of $P$ to $D$. For the coupling
we complete it by sending the missing mass to $\dagger$ and extend the link
at $\dagger$ by the point mass at $\dagger$. In the conservative case $D=V$
this is the marked evolving-set construction of~\cite[Section~2.5]{AndersenPeres2009}.

\begin{lemma}[Process-level punctured intertwining]\label{lem:process-intertwining}
The three kernels satisfy
\begin{equation}\label{eq:matrix-intertwining}
  \wh K_D\Lambda=\Lambda P_D.
\end{equation}
Consequently, for each $x\in D$ there exists a Markovian coupling
$(\wh S_n^D,\wh X_n^D)_{n\geq0}$ started at $(\{x\},x)$ such that
\begin{enumerate}[label=\textup{(\roman*)}]
  \item $\wh S_n^D$ has transition kernel $\wh K_D$,
  \item conditionally on $\wh S_n^D=A\ne\dagger$, the marked vertex has law
        $\Pp(\wh X_n^D=y\mid\wh S_n^D=A)=\Lambda(A,y)$,
  \item $\wh X_n^D$ has the law of the $P$-chain killed on exiting $D$.
\end{enumerate}
The same assertions hold in continuous time when both coordinates run on the same
independent Poisson clock of rate one.
\end{lemma}

The law of the transformed set on its surviving states is related to the ordinary
law by a change of measure. For $\varnothing\ne A\subseteq D$ and $t\geq0$,
\begin{equation}\label{eq:transformed-marginal}
  \wh\Pp^D_{\{x\}}(\wh S_t^D=A)
     =\Pp^D_{\{x\}}(S_t^D=A)\frac{\pi(A)}{\pi(x)} .
\end{equation}
Let us now sum over the sets containing $y$. Multiplying each summand by
$\Lambda(A,y)$ cancels the factor $\pi(A)$ and gives
\begin{equation}\label{eq:marked-marginal}
  \Pp_x(\wh X_t^D=y)
  =\frac{\pi(y)}{\pi(x)}\Pp^D_{\{x\}}(y\in S_t^D)
  =\Pp_x(X_t=y,\ t<\tau_{D^c}),\qquad y\in D .
\end{equation}

\subsection{The punctured Wilson--evolving-set object}

Start the marked pair at $(\{x\},x)$ and run it up to an independent exponential
time $T_q$. By Equality~\eqref{eq:marked-marginal}, the terminal position has law
on $D$ given by
\begin{equation}\label{eq:final-point-projection}
  \Pp_x(\wh X_{T_q}^D=y)
    =\frac{\pi(y)}{\pi(x)}\Pp^D_{\{x\}}(y\in S_{T_q}^D).
\end{equation}
The probability on the right is under the ordinary evolving-set law. Summing over
$y\in D$ gives the probability that the marked trajectory is still in $D$ at $T_q$:
\begin{equation}\label{eq:survival-projection}
  \Pp_x(T_q<\tau_{D^c})
   =\frac{1}{\pi(x)}\Ee^D_{\{x\}}[\pi(S_{T_q}^D)].
\end{equation}
Taking the complement yields
\begin{equation}\label{eq:exit-projection}
  \Pp_x(\tau_{D^c}\leq T_q)
   =1-\frac{1}{\pi(x)}\Ee^D_{\{x\}}[\pi(S_{T_q}^D)].
\end{equation}

To apply this construction in Wilson's algorithm, let $W$ be the vertex set of the
part of the forest already explored. For $x\notin W$, run the marked trajectory in
$D=V\setminus W$. If $T_q$ occurs before exit, the marked vertex at $T_q$ is the root
of the new component. If the walk exits first, its branch must attach to a vertex
of $W$. The cemetery state tells us that the walk has left $D$. To recover the
vertex it hit, keep the last vertex $u$ inside $D$ and append $z\in W$ with
conditional probability $P(u,z)/P(u,W)$, for $P(u,W)>0$. Erase the loops of the
resulting path and add the branch to the forest.

\newpage
\section{Proofs}\label{sec:proofs}

\subsection{Proof of the punctured intertwining}
\begin{proof}[Proof of Lemma~\ref{lem:punctured-intertwining}]
Let $(Y_n)_{n\geq0}$ be the discrete-time chain with kernel $P$, and let
$\tau_{D^c}^{\mathrm d}=\inf\{n\geq0:Y_n\notin D\}$. We first prove
\eqref{eq:punctured-intertwining} at integer times. Fix $A\subseteq D$
and $y\in D$. By the uniform threshold in
\eqref{eq:punctured-evolving-set-definition}, the probability that the next
set contains $y$, starting from a set $B\subseteq D$, is
\[
    \frac{Q_D(B,y)}{\pi(y)}
    =\frac{1}{\pi(y)}\sum_{z\in B}\pi(z)P(z,y).
\]
We claim that for every $n\geq0$,
\[
  \Pp_A^D(y\in S_n^D)
  =\frac{1}{\pi(y)}\sum_{x\in A}\pi(x)
    \Pp_x(Y_n=y,\ n<\tau_{D^c}^{\mathrm d}).
\]
At time zero both sides equal $\1_{\{y\in A\}}$. Suppose the claim
holds at time $n$ for every $y\in D$. Conditioning on $S_n^D$ gives
\begin{align*}
  \Pp_A^D(y\in S_{n+1}^D)
  &=\frac{1}{\pi(y)}\sum_{z\in D}\pi(z)P(z,y)\Pp_A^D(z\in S_n^D)\\
  &=\frac{1}{\pi(y)}\sum_{x\in A}\pi(x)\sum_{z\in D}
      \Pp_x(Y_n=z,\ n<\tau_{D^c}^{\mathrm d})P(z,y)\\
  &=\frac{1}{\pi(y)}\sum_{x\in A}\pi(x)
      \Pp_x(Y_{n+1}=y,\ n+1<\tau_{D^c}^{\mathrm d}).
\end{align*}
The second equality uses the induction hypothesis. In the last equality
the intermediate vertex must lie in $D$, since the walk stops on leaving $D$.

Now let $N_t$ be the number of rings of an independent rate-one Poisson clock
up to time $t$. Averaging the discrete-time identity over $N_t$ yields
\begin{align*}
 \Pp_A^D(y\in S_t^D)
 &=\frac{1}{\pi(y)}\sum_{x\in A}\pi(x)
   \sum_{n\geq0}\Pp(N_t=n)\Pp_x(Y_n=y,\ n<\tau_{D^c}^{\mathrm d})\\
 &=\frac{1}{\pi(y)}\sum_{x\in A}\pi(x)\Pp_x(X_t=y,\ t<\tau_{D^c}),
\end{align*}
as required.
\end{proof}

\subsection{Proof of the process-level coupling}
\begin{proof}[Proof of Lemma~\ref{lem:process-intertwining}]
Fix $\varnothing\ne A\subseteq D$ and $y\in D$. By
\eqref{eq:transformed-kernel-definition} and \eqref{eq:link-definition},
\begin{align*}
 (\wh K_D\Lambda)(A,y)
 &=\sum_{\varnothing\ne B\subseteq D}K_D(A,B)
   \frac{\pi(B)}{\pi(A)}\frac{\pi(y)}{\pi(B)}\1_{\{y\in B\}}\\
 &=\frac{\pi(y)}{\pi(A)}\Pp_A^D(y\in S_1^D)
 =\frac{1}{\pi(A)}\sum_{x\in A}\pi(x)P(x,y)
 =(\Lambda P_D)(A,y).
\end{align*}
Here $S_1^D$ denotes one step of the discrete chain, so the third equality
follows from its threshold rule.

To include the exit from $D$, complete $P_D$ and $\wh K_D$ by putting the
missing mass at an absorbing state $\dagger$. Extend the link by
$\Lambda(\dagger,\dagger)=1$ and $\Lambda(A,\dagger)=0$ for a live set $A$.
Equality~\eqref{eq:matrix-intertwining} extends to the cemetery column by
summing the live columns and taking complements. On pairs $(A,x)$ with
$\Lambda(A,x)>0$, put
\[
  \mathcal K((A,x),(B,y))
  =P_D(x,y)\,
   \frac{\wh K_D(A,B)\Lambda(B,y)}{(\Lambda P_D)(A,y)}
\]
when the denominator is positive. Set this entry to zero when the
denominator vanishes. Indeed, on the indicated support,
$(\Lambda P_D)(A,y)=0$ implies $P_D(x,y)=0$.

Summing in $B$ and using the intertwining gives
$\sum_B\mathcal K((A,x),(B,y))=P_D(x,y)$. Thus the marked point follows
the killed walk even when its present set is known. To obtain the law
of the set, start with the singleton $\{x\}$ marked at $x$. Suppose that
conditionally on the set history up to the current set $A$, the mark has
law $\Lambda(A,\cdot)$. Then
\[
  \sum_z\Lambda(A,z)\mathcal K((A,z),(B,y))
   =\wh K_D(A,B)\Lambda(B,y).
\]
For a positive denominator the identity follows by cancellation. For a
zero denominator both sides vanish by the intertwining. Summing over $y$
gives the transition probability $\wh K_D(A,B)$ for the set. Given the
enlarged set history ending at $B$, the new mark has law $\Lambda(B,\cdot)$.
Induction proves the discrete-time coupling. Running both coordinates on
the same independent Poisson clock gives the continuous-time assertion.
\end{proof}

Along a path of non-empty sets starting at $\{x\}$ and ending at $A$, the
successive mass ratios in \eqref{eq:transformed-kernel-definition} telescope to
$\pi(A)/\pi(x)$. Averaging over the Poisson clock proves
\eqref{eq:transformed-marginal}. For $x,y\in D$, condition now on the set
that carries the mark at time $t$. We obtain
\begin{align*}
  \Pp_x(\wh X_t^D=y)
  &=\sum_{\varnothing\ne A\subseteq D}
       \wh\Pp_{\{x\}}^D(\wh S_t^D=A)\Lambda(A,y)\\
  &=\frac{\pi(y)}{\pi(x)}\Pp_{\{x\}}^D(y\in S_t^D)
   =\Pp_x(X_t=y,\ t<\tau_{D^c}),
\end{align*}
where the last equality is \eqref{eq:singleton-intertwining}.

For the survival probability, multiply \eqref{eq:singleton-intertwining}
by $\pi(y)$ and sum over $y\in D$. This gives
\begin{equation}\label{egalitpreuve}
  \Ee_{\{x\}}^D[\pi(S_t^D)]=\pi(x)\Pp_x(t<\tau_{D^c}),
  \qquad x\in D,\quad t\geq0.
\end{equation}
Since $T_q$ is independent and has density $qe^{-qt}$,
\begin{align*}
  \Ee_{\{x\}}^D[\pi(S_{T_q}^D)]
  &=q\int_0^\infty e^{-qt}\Ee_{\{x\}}^D[\pi(S_t^D)]\dd t\\
  &=\pi(x)q\int_0^\infty e^{-qt}\Pp_x(t<\tau_{D^c})\dd t
   =\pi(x)\Pp_x(T_q<\tau_{D^c}).
\end{align*}
Dividing by $\pi(x)$ gives \eqref{eq:survival-projection}. Its complement
is \eqref{eq:exit-projection}.

\subsection{Proofs of the root identities}\label{sec:proof-root-identities}

We first check the forest law used in the exploration. At each ring of a
rate-$(1+q)$ clock the walk goes from $u$ to $v\in V$ with probability
$B(u,v)=P(u,v)/(1+q)$, or to the cemetery state with probability $q/(1+q)$.
Suppose the branch starts at $u$, and let $D$ be the set of vertices still
unexplored. For its loop-erasure to take a prescribed first edge, the walk
may first make any number of steps in $D$ that bring it back to $u$.
After its last visit to $u$ it must take that edge and avoid $u$ until
the branch ends. Summing over these possible initial paths gives the factor
\[
 \sum_{m\ge0} B_D^m(u,u)=(I-B_D)^{-1}(u,u)
 =\frac{\det(I-B_{D\setminus\{u\}})}{\det(I-B_D)},
\]
where $B_D$ is the restriction to $D$. Its row sums are at most $1/(1+q)$,
so the series converges. The last equality is the diagonal cofactor formula. The first edge contributes its transition probability.
If its endpoint is still unexplored, repeat the calculation there, excluding
paths that return to $u$, hence using the restriction to $D\setminus\{u\}$. If it is the cemetery state or a vertex of the previous
forest, the branch ends.

For a prescribed rooted forest, this calculation removes each vertex once.
Across successive branches the determinant ratios telescope to
$1/\det(I-B_V)$, with the empty determinant equal to one. The remaining
factors are $P(u,p_F(u))/(1+q)$ at non-roots and $q/(1+q)$ at roots.
Thus the probability is proportional to
$q^{|\R(F)|}\prod_{u\notin\R(F)}P(u,p_F(u))$, which gives
\eqref{eq:forest-law} in every exploration order.

\begin{proof}[Proof of the root-law part of Theorem~\ref{thm:intro-roots}]
Start Wilson's algorithm at $x$. At this first step $D=V$. The root of the
component of $x$ is the marked position $\wh X_{T_q}^V$ at killing.
Thus \eqref{eq:final-point-projection} gives
\[
   \Pp(\rho_q(x)=r)=\Pp_x(\wh X_{T_q}^V=r)
   =\frac{\pi(r)}{\pi(x)}\Pp_{\{x\}}(r\in S_{T_q}),
   \qquad r\in V.
\]
Now fix $r$. Summing this equality over the starting point with weight
$\pi(x)$ gives
\[
  \Ee[\pi(B_q(r))]=\pi(r)\sum_{x\in V}\Pp_{\{x\}}(r\in S_{T_q}).
\]
For each $t\geq0$, \eqref{eq:singleton-intertwining} in $V$ and the
invariance of $\pi$ give
\[
  \sum_{x\in V}\Pp_{\{x\}}(r\in S_t)
  =\frac{1}{\pi(r)}\sum_{x\in V}\pi(x)P_t(x,r)=1.
\]
Averaging at $T_q$ gives $\Ee[\pi(B_q(r))]=\pi(r)$.
\end{proof}

\begin{proof}[Proof of the multi-root part of Theorem~\ref{thm:intro-roots}]
Run Wilson's algorithm in an order beginning with $x_1,\ldots,x_k$. The
chain rule gives
\[
  \Pp(A_k\subseteq\Root)
  =\prod_{i=1}^k \Pp(x_i\in\Root\mid A_{i-1}\subseteq\Root).
\]
Let us identify the domain for the $i$th walk under this conditioning.
For $x_1$ to be a root, the walk from $x_1$ must be killed at $x_1$.
Its chronological loop-erasure is then the singleton $\{x_1\}$.
Applying the same argument at each of the next steps shows that on
$\{A_{i-1}\subseteq\Root\}$ the first $i-1$ branches are the
singletons $\{x_1\},\ldots,\{x_{i-1}\}$. In particular, the set
already constructed when the walk from $x_i$ starts is exactly $A_{i-1}$.

The new walk must therefore be killed at $x_i$ before hitting $A_{i-1}$.
In \eqref{eq:final-point-projection} take $D=D_{i-1}=V\setminus A_{i-1}$
and $y=x_i$. The factor $\pi(y)/\pi(x_i)$ is then one, and
\[
  \Pp(x_i\in\Root\mid A_{i-1}\subseteq\Root)
  =\Pp_{\{x_i\}}^{D_{i-1}}\bigl(x_i\in S_{T_q}^{D_{i-1}}\bigr).
\]
Substitution into the product gives \eqref{eq:intro-multiroot}.
\end{proof}

\begin{proof}[Proof of Corollary~\ref{cor:abstract-transfer}]
Fix $x\in D$ and condition on $T_q$. By
\eqref{eq:abstract-transfer-assumption},
\[
 \Pp_{\{x\}}^D(x\in S_{T_q}^D)
 =\int_0^\infty qe^{-qt}\Pp_{\{x\}}^D(x\in S_t^D)\dd t
 \leq\int_0^\infty qe^{-qt}G_D(t,x)\dd t.
\]
Apply this estimate in \eqref{eq:intro-multiroot}, with $x=x_i$ and
$D=D_{i-1}$ for each $i$.
\end{proof}

\subsection{Proof of the determinantal comparison}
\begin{proof}[Proof of Proposition~\ref{prop:determinantal-comparison}]
Put $M=qI+L$ and let $M_D=qI_D+L_D$ be its principal restriction to
$D$. Each $M_D$ is invertible for $D\ne\varnothing$, since $q>0$
and $L_D$ is non-negative in $L^2(D,\pi)$. By the factorization
\eqref{eq:intro-multiroot} and the diagonal case of
\eqref{eq:punctured-resolvent},
\[
  \Pp(A\subseteq\Root)
   =\prod_{i=1}^k q\bigl(M_{D_{i-1}}^{-1}\bigr)(x_i,x_i).
\]
For each $i$, deleting $x_i$ from $D_{i-1}$ gives $D_i$. Hence Cramer's
rule gives $\bigl(M_{D_{i-1}}^{-1}\bigr)(x_i,x_i)=
\det M_{D_i}/\det M_{D_{i-1}}$, with $\det M_{\varnothing}=1$. The
product telescopes from $D_0=V$ to $D_k=V\setminus A$:
\[
 \Pp(A\subseteq\Root)
   =q^k\prod_{i=1}^k\frac{\det M_{D_i}}{\det M_{D_{i-1}}}
   =q^k\frac{\det M_{V\setminus A}}{\det M}.
\]
To identify this ratio, let $J_A$ be the diagonal matrix with entries
$\1_{\{x\in A\}}$. In $\det(M+tJ_A)$, expansion in the columns indexed by
$A$ gives $\det M_{V\setminus A}$ as the coefficient of $t^k$.
Writing the same polynomial as $\det M\det(I+tM^{-1}J_A)$ gives
$\det M\det((M^{-1})[A])$ for this coefficient. The ratio above is
therefore $\det((M^{-1})[A])$. Since $K_q=qM^{-1}$, we obtain
$\Pp(A\subseteq\Root)=\det K_q[A]$.
\end{proof}

\subsection{Proofs of the quantitative root bounds}
\begin{proof}[Proof of Corollary~\ref{cor:dirichlet-transfer}]
Write $P_t^D=e^{-tL_D}$ for the semigroup of the chain killed on leaving
$D$. Fix $x\in D$. Applying \eqref{eq:punctured-intertwining} with
$A=\{x\}$ and $y=x$ gives $\Pp_{\{x\}}^D(x\in S_t^D)=P_t^D(x,x)$.
Since $L_D$ is self-adjoint in $L^2(D,\pi)$ and its spectrum is contained
in $[\lambda_1(D),\infty)$, the spectral theorem gives
\[
 P_t^D(x,x)
 =\frac{\langle\1_{\{x\}},e^{-tL_D}\1_{\{x\}}\rangle_\pi}{\pi(x)}
 \leq e^{-\lambda_1(D)t}
       \frac{\|\1_{\{x\}}\|_{L^2(\pi)}^2}{\pi(x)}
 =e^{-\lambda_1(D)t}.
\]
At the independent time $T_q$ the same diagonal is bounded on average by
\[
 \Pp_{\{x\}}^D(x\in S_{T_q}^D)
 \leq q\int_0^\infty e^{-(q+\lambda_1(D))t}\dd t
 =\frac{q}{q+\lambda_1(D)}.
\]
Applying this estimate to the factors of \eqref{eq:intro-multiroot}
gives \eqref{eq:dirichlet-root-bound}.

Let us prove the estimate on $\lambda_1(D)$ used in the second bound.
Take a non-zero real function $f$ on $D$, extend it by zero to $V$, and
put $g=|f|$. Then $\cE(g,g)\leq\cE(f,f)$ and $\|g\|_{L^2(\pi)}=\|f\|_{L^2(\pi)}$.
For $t\geq0$ let $A_t=\{u:g(u)^2>t\}$. Each $A_t$ lies in $D$, so
$Q(A_t,A_t^c)\geq\phi_D\pi(A_t)$. Since
$\int_0^\infty\pi(A_t)\dd t=\|g\|_{L^2(\pi)}^2$, integration gives
\[
 \phi_D\|g\|_{L^2(\pi)}^2
 \leq \int_0^\infty Q(A_t,A_t^c)\dd t
 =\frac12\sum_{u,v\in V}c(u,v)\,|g(u)^2-g(v)^2|.
\]
Indeed, an edge with $g(u)>g(v)$ crosses the level set for
$g(v)^2\leq t<g(u)^2$, and $c(u,v)=c(v,u)$. Factoring the difference
of squares and applying Cauchy--Schwarz, we obtain
\[
 \phi_D^2\|g\|_{L^2(\pi)}^4
 \leq \cE(g,g)\,\frac12\sum_{u,v\in V}c(u,v)(g(u)+g(v))^2
 \leq 2\cE(g,g)\|g\|_{L^2(\pi)}^2.
\]
The last inequality uses $(a+b)^2\leq2a^2+2b^2$ and
$\sum_v c(u,v)=\pi(u)$. Dividing by $\|g\|_{L^2(\pi)}^4$ and taking
the infimum over $f$ yields $\lambda_1(D)\geq\phi_D^2/2$. Substitution
in \eqref{eq:dirichlet-root-bound} proves \eqref{eq:cheeger-root-bound}.
\end{proof}

\begin{proof}[Proof of Theorem~\ref{thm:spectral-domination}]
Fix $x\in D$. To bound a punctured factor using the spectral gap of the
whole chain, first observe that
\[
 P_t^D(x,x)
 =\Pp_x(X_t=x,\ t<\tau_{D^c})\leq\Pp_x(X_t=x)=P_t(x,x).
\]
Take an orthonormal eigenbasis $(\varphi_j)_{1\leq j\leq |V|}$ of
$L$ in $L^2(\bar\pi)$, with $\varphi_1\equiv1$ and eigenvalues
$0=\lambda_1<\lambda_2=\gamma\leq\cdots\leq\lambda_{|V|}$.
On the diagonal the spectral expansion reads
\[
 P_t(x,x)=\bar\pi(x)+\bar\pi(x)\sum_{j=2}^{|V|}
           e^{-\lambda_jt}\varphi_j(x)^2.
\]
At $t=0$ the sum with its prefactor is $1-\bar\pi(x)$. Since every
$\lambda_j$ in the sum is at least $\gamma$, we obtain
\[
 P_t^D(x,x)\leq P_t(x,x)
 \leq \bar\pi(x)+(1-\bar\pi(x))e^{-\gamma t}.
\]
Integrating against $qe^{-qt}\dd t$ gives
\[
 \Pp_{\{x\}}^D(x\in S_{T_q}^D)
 \leq \bar\pi(x)+(1-\bar\pi(x))\frac{q}{q+\gamma}=\beta_q(x).
\]
We now have two estimates for the same factor. The preceding one uses
$\gamma$ on $V$, while the proof of Corollary~\ref{cor:dirichlet-transfer}
gives $q/(q+\lambda_1(D))$ on the particular domain $D$. In each factor
of \eqref{eq:intro-multiroot} we may take the smaller of these bounds.
This gives \eqref{eq:combined-root-product}. Keeping only $\beta_q(x_i)$
in each factor gives \eqref{eq:root-product-domination}.
\end{proof}

\begin{proof}[Proof of Proposition~\ref{prop:root-moments}]
For $U\subseteq V$, the falling factorial $(N_q(U))_k$ counts the ordered
$k$-tuples of distinct roots in $U$. Therefore
\[
 \Ee[(N_q(U))_k]
 =\sum_{\substack{x_1,\ldots,x_k\in U\\\mathrm{distinct}}}
   \Pp(x_1,\ldots,x_k\in\Root)
 =\sum_{\substack{x_1,\ldots,x_k\in U\\\mathrm{distinct}}}
   \prod_{i=1}^k
   \Pp_{\{x_i\}}^{V\setminus A_{i-1}}
      (x_i\in S_{T_q}^{V\setminus A_{i-1}}),
\]
where the second equality uses \eqref{eq:intro-multiroot}.
Taking $U=V$ gives the formula for the total number of roots.
For the void probability, expand the product
\[
 \1_{\{\Root\cap U=\emptyset\}}
 =\prod_{x\in U}(1-\1_{\{x\in\Root\}})
 =\sum_{A\subseteq U}(-1)^{|A|}\1_{\{A\subseteq\Root\}}
\]
and take expectations.

For the basin formulas, sum \eqref{eq:intro-root-law} over $r\in U$:
\[
 \Pp(\rho_q(x)\in U)
 =\frac{1}{\pi(x)}\sum_{r\in U}\pi(r)\Pp_{\{x\}}(r\in S_{T_q})
 =\frac{1}{\pi(x)}\Ee_{\{x\}}[\pi(S_{T_q}\cap U)].
\]
The basins are disjoint, so that \eqref{eq:intro-basin-law} gives
$\Ee[\pi(B_q(U))]=\sum_{r\in U}\Ee[\pi(B_q(r))]=\pi(U)$.
Finally, each tree has one fewer edge than vertices. Summing over its
components gives $|E(F_q)|=|V|-N_q$ for every forest $F_q$, and the
last identity follows by taking expectations.
\end{proof}

\begin{proof}[Proof of Corollary~\ref{cor:root-concentration}]
Expand the product, recalling that every $s_x$ is non-negative. We can
then apply \eqref{eq:root-product-domination} term by term:
\begin{align*}
  \Ee\!\left[\prod_{x\in U}
       (1+s_x\1_{\{x\in\Root\}})\right]
   &=\sum_{A\subseteq U}\left(\prod_{x\in A}s_x\right)\Pp(A\subseteq\Root)\\
   &\leq\sum_{A\subseteq U}\prod_{x\in A}s_x\beta_q(x)
    =\prod_{x\in U}(1+s_x\beta_q(x)).
\end{align*}
Take $s_x=e^\theta-1$ with $\theta\geq0$. The product inside the expectation
is $e^{\theta N_q(U)}$. The bound $1+u\leq e^u$ therefore gives
\[
 \Ee[e^{\theta N_q(U)}]
 \leq\prod_{x\in U}\bigl(1+\beta_q(x)(e^\theta-1)\bigr)
 \leq\exp\bigl(\mu_q(U)(e^\theta-1)\bigr).
\]
For $\delta>0$, Markov's inequality now yields
\[
 \Pp(N_q(U)\geq(1+\delta)\mu_q(U))
 \leq\exp\!\left[\mu_q(U)\bigl(e^\theta-1-\theta(1+\delta)\bigr)\right].
\]
Choosing $\theta=\log(1+\delta)$ proves \eqref{eq:root-chernoff}.
Finally, summing \eqref{eq:root-product-domination} over the ordered
$k$-tuples of distinct vertices in $U$ gives
\[
 \Ee[(N_q(U))_k]
 \leq k!\,e_k\bigl((\beta_q(x))_{x\in U}\bigr)
 \leq\left(\sum_{x\in U}\beta_q(x)\right)^k=\mu_q(U)^k .
\]
\end{proof}
\subsection{Proofs of the hitting and forest estimates}

\begin{proof}[Proof of the hitting identity in Theorem~\ref{thm:intro-hitting}]
Fix $x\notin U$ and take $D=V\setminus U$. The first exit from $D$ is precisely the
first visit to $U$. Applying Equality~\eqref{eq:exit-projection} to this domain gives
\[
  h_q(x,U)
  =\Pp_x(\tau_{D^c}\le T_q)
  =1-\frac{1}{\pi(x)}
     \Ee_{\{x\}}^{D}\!\left[\pi\bigl(S_{T_q}^{D}\bigr)\right].
\]
\end{proof}

\begin{proof}[Proof of the forest bounds in Theorem~\ref{thm:intro-hitting}]
Run Wilson's algorithm first from $x$. A vertex $z$ which belongs to the
loop-erased branch $L_x$ was visited by the walk before it was killed. Consequently,
\[
   \Pp(z\in L_x)\leq \Pp_x(\tau_z\leq T_q)=h_q(x,z),
\]
which proves Inequality~\eqref{eq:intro-branch-hitting}.

Condition now on $L_x$ and run an independent walk from $y$, with its own
exponential killing time of rate $q$. If $y\notin L_x$, this walk joins
the component of $x$ precisely when it hits $L_x$ before it is killed.
If $y\in L_x$, the connection has already been made and $\tau_{L_x}=0$.
Thus in both cases
\[
   \Pp(x\leftrightarrow_q y\mid L_x)=h_q(y,L_x).
\]
This is the two-branch representation of Wilson's algorithm, also used for the
probability of separation in~\cite[Section~1.4]{AvenaGaudilliereMilanesiQuattropani2022}.
For a fixed branch, bound the probability of hitting it by the sum of the
hitting probabilities of its vertices. Averaging over $L_x$ gives
\[
 \begin{aligned}
 \Pp(x\leftrightarrow_q y)
  &=\Ee[h_q(y,L_x)]\\
  &\leq\Ee\!\left[\sum_{z\in L_x}h_q(y,z)\right]
   =\sum_{z\in V}\Pp(z\in L_x)h_q(y,z)\\
  &\leq\sum_{z\in V}h_q(x,z)h_q(y,z),
 \end{aligned}
\]
where the last inequality follows from Inequality~\eqref{eq:intro-branch-hitting}.
\end{proof}

Let us look at the two inequalities in this calculation. In the first one,
a walk from $y$ which visits two vertices of $L_x$ before killing is counted
at least twice if we follow it until killing. In Wilson's algorithm its
first visit to the branch already settles the question. In the second inequality,
the walk from $x$ may visit $z$ and subsequently erase every visit to $z$.
Such a trajectory contributes to $h_q(x,z)$ even though $z$ is absent from $L_x$.

\begin{lemma}[Resolvent representation of hitting]\label{lem:resolvent-hitting-ratio}
Let $G_q(x,y)=\int_0^\infty e^{-qt}P_t(x,y)\dd t$. For all $x,y\in V$,
\begin{equation}\label{eq:resolvent-hitting-ratio}
  h_q(x,y)=\frac{G_q(x,y)}{G_q(y,y)}.
\end{equation}
\end{lemma}

\begin{proof}
Write $\tau_y=\tau_{\{y\}}$. Independence of the killing time gives
$h_q(x,y)=\Ee_x[e^{-q\tau_y}]$, with $e^{-q\tau_y}=0$ on $\{\tau_y=\infty\}$.
By Tonelli's theorem,
\[
 G_q(x,y)=\Ee_x\!\left[\int_0^\infty e^{-qt}\1_{\{X_t=y\}}\dd t\right].
\]
The walk spends no time at $y$ before $\tau_y$. By the strong Markov property at
$\tau_y$, the expected discounted occupation after that time, measured from
$\tau_y$, is $G_q(y,y)$. Thus $G_q(x,y)=\Ee_x[e^{-q\tau_y}]G_q(y,y)$.
Since $G_q(y,y)>0$, we may divide by it.
\end{proof}

\begin{proof}[Proof of Theorem~\ref{thm:intro-gaussian}]
The walk from $z$ is still at $z$ if the rate-one Poisson clock has not rung.
Thus $P_t(z,z)\ge e^{-t}$ and, for every $q>0$,
\begin{equation}\label{eq:diagonal-resolvent-lower}
 G_q(z,z)\ge \int_0^\infty e^{-(1+q)t}\dd t=\frac1{1+q}.
\end{equation}
By Equality~\eqref{eq:resolvent-hitting-ratio} we have
$h_q(x,z)\le(1+q)G_q(x,z)$. Insert this bound in
Inequality~\eqref{eq:intro-connectivity}. Since $P$ is symmetric,
$G_q(y,z)=G_q(z,y)$, so the sum over $z$ gives
\[
 \Pp(x\leftrightarrow_q y)
 \le (1+q)^2\sum_{z\in V}G_q(x,z)G_q(z,y)
 = (1+q)^2(G_q^2)(x,y).
\]
The semigroup property and Tonelli's theorem give
\[
 (G_q^2)(x,y)=\int_0^\infty\!\int_0^\infty e^{-q(s+t)}P_{s+t}(x,y)\dd s\dd t
              =\int_0^\infty u e^{-qu}P_u(x,y)\dd u.
\]
To prove \eqref{eq:intro-component-gaussian}, sum over $y$. At each time $u$,
$\sum_y P_u(x,y)=1$, so the integral is $q^{-2}$ after summation.

We now keep $x$ and $y$ fixed and assume~\eqref{eq:intro-gaussian-bound}.
For $j=1,2$ the integrals to be estimated are
$\int_0^\infty t^{j-1}e^{-qt}P_t(x,y)\dd t$. They give $G_q$ and $G_q^2$,
respectively. The contribution of $C_H/n$ is $C_H/(nq^j)$ in each case.
Let $r=d_G(x,y)\ge1$, take $s=1+t$ and use $q\le1$.
The contribution of the first term in~\eqref{eq:intro-gaussian-bound} is at most
\[
 C\int_1^\infty s^{j-1-\nu/2}\exp(-qs-c_Hr^2/s)\dd s
 \le C e^{-c\sqrt q\,r}
     \int_0^\infty s^{j-1-\nu/2}\exp(-c_Hr^2/(2s))\dd s.
\]
Here $qs+c_Hr^2/(2s)\ge\sqrt{2c_Hq}\,r$. The change of variable $s=r^2v$
gives a factor $r^{2j-\nu}$ in the last integral. Its remaining integral in
$v$ is finite since $\nu>2j$. At $r=0$ we use directly
$\int_0^\infty t^{j-1}(1+t)^{-\nu/2}\dd t<\infty$.
Thus we can replace $r^{2j-\nu}$ by $(1+r)^{2j-\nu}$ in a bound valid
for all $x,y$. The case $j=1$, together with the diagonal bound
\eqref{eq:diagonal-resolvent-lower}, gives \eqref{eq:intro-hitting-gaussian}.
Applying \eqref{eq:squared-resolvent-connectivity} with the estimate for
$j=2$ gives \eqref{eq:intro-connectivity-gaussian}. It is this last integral
that requires $\nu>4$.
\end{proof}

\begin{proof}[Proof of Proposition~\ref{prop:components}]
Write $\pi(C_q(x))=\sum_y\pi(y)\1_{\{x\leftrightarrow_q y\}}$.
Taking expectations and using Inequality~\eqref{eq:intro-connectivity} gives
\[
 \Ee[\pi(C_q(x))]
 \le \sum_{z\in V}h_q(x,z)\sum_{y\in V}\pi(y)h_q(y,z).
\]
This is \eqref{eq:mean-component}. Multiplying by $\pi(x)$ and summing
over $x$ gives \eqref{eq:susceptibility}. To bound the largest component,
observe that
\[
  M_q^2\le\sum_{C\text{ component}}\pi(C)^2
  =\sum_{x,y\in V}\pi(x)\pi(y)\1_{\{x\leftrightarrow_q y\}}.
\]
The expectation of the right-hand side is $\chi_q$. Markov's inequality
gives $\Pp(M_q\ge a)\le \chi_q/a^2$ for $a>0$.
Finally, $C_q(x)$ meets $U$ exactly when $x$ is connected to some
$y\in U$. Summing \eqref{eq:intro-connectivity} over $y\in U$ gives
\eqref{eq:component-hits-set}.
\end{proof}

\begin{proof}[Proof of Proposition~\ref{prop:spatial}]
A vertex $z$ belongs to the range of the raw walk precisely when
$\tau_z\le T_q$. Consequently,
\[
  \Ee[\pi(R_x^{\mathrm{raw}})]=\sum_{z\in V}\pi(z)h_q(x,z).
\]
Since $L_x\subset R_x^{\mathrm{raw}}$, this also bounds the expected mass
of the branch. If its radius exceeds $R$, the walk from $x$ must have
left $B(x,R)$ before killing. Thus
$\Pp(\operatorname{rad}(L_x)>R)\le h_q(x,B(x,R)^c)$.
Taking $U=B(x,R)^c$ in Equality~\eqref{eq:intro-hitting-identity}
gives the bound in terms of the evolving set.

For the depth, let $J_t$ count the events of the rate-one Poisson clock up to
time $t$. The branch has at most $J_{T_q}$ edges. Conditioning on $T_q$
gives
\[
  \Pp(J_{T_q}=m)
  =\int_0^\infty qe^{-qt}e^{-t}\frac{t^m}{m!}\dd t
  =\frac{q}{(1+q)^{m+1}},\qquad m\ge0 .
\]
Summing this law gives $\Pp(J_{T_q}\ge m)=(1+q)^{-m}$. In turn,
summing over $m\ge1$ gives $\Ee[\operatorname{depth}_q(x)]\le 1/q$.

If $u,v\in C_q(x)$ satisfy $d(u,v)\ge R$, then at least one of them is at
distance $R/2$ or more from $x$. We bound the probability by summing
\eqref{eq:intro-connectivity} over all $y$ at distance at least $R/2$
from $x$. This proves \eqref{eq:local-diameter}. For
\eqref{eq:global-diameter}, sum instead over all pairs at distance at
least $R$, since any component of diameter at least $R$ contains such a pair.
\end{proof}

\subsection{Proof of the escape-profile bound}
\begin{proof}[Proof of Theorem~\ref{thm:escape-profile}]
Fix $x\notin U$, let $D=V\setminus U$ and put
$m(t)=\Ee^D_{\{x\}}[\pi(S_t^D)]$.
We first compute the expected mass lost in one update. Given the current
set $A\subset D$, denote by $B$ the set after that update. By
\eqref{eq:punctured-evolving-set-definition},
\[
  \Ee[\pi(B)\mid A]
  =\sum_{y\in D}\pi(y)\Pp(y\in B\mid A)
  =\sum_{y\in D}Q_D(A,y)
  =\pi(A)-Q(A,U).
\]
The clock has rate one, so the continuous-time generator applied to the
mass function at $A$ is $-Q(A,U)$. Taking its expectation gives
\begin{equation}\label{eq:escape-mass-derivative}
 m'(t)=-\Ee^D_{\{x\}}[Q(S_t^D,U)],\qquad t\ge0.
\end{equation}
For every $A\subset D$, $Q(A,U)\le\widehat\Psi_U(\pi(A))$ by
\eqref{eq:escape-profile-definition}. Concavity now gives the comparison
in the required direction:
\[
  m'(t)\ge-\Ee^D_{\{x\}}[\widehat\Psi_U(\pi(S_t^D))]
          \ge-\widehat\Psi_U(m(t)).
\]
Since $D$ is finite, $\widehat\Psi_U$ is piecewise linear and Lipschitz
on $[0,\pi(D)]$. The initial values agree, $m(0)=u_{x,U}(0)=\pi(x)$,
so comparison with the solution of \eqref{eq:escape-comparison-ode}
gives $m(t)\ge u_{x,U}(t)$. Averaging the mass at the independent
exponential time and applying \eqref{eq:exit-projection} yields
\[
 h_q(x,U)
 =1-\frac{q}{\pi(x)}\int_0^\infty e^{-qt}m(t)\dd t
 \le 1-\frac{q}{\pi(x)}\int_0^\infty e^{-qt}u_{x,U}(t)\dd t
 =H_q(x,U).
\]
\end{proof}

\begin{proof}[Proof of Corollary~\ref{cor:linear-escape}]
With the same notation, we have
$Q(A,U)=\sum_{a\in A}\pi(a)P(a,U)\le\lambda_U\pi(A)$
for every $A\subset D$. Inserting this estimate in
Equality~\eqref{eq:escape-mass-derivative} gives $m'(t)\ge-\lambda_Um(t)$,
whence $m(t)\ge\pi(x)e^{-\lambda_Ut}$. Substituting in the integral for
$h_q(x,U)$ gives
\[
 h_q(x,U)\le 1-q\int_0^\infty e^{-(q+\lambda_U)t}\dd t
 =\frac{\lambda_U}{q+\lambda_U}.
\]
\end{proof}
\pagebreak[2]
\section{Evolving sets and the resolvent}\label{sec:resolvent}

\subsection{One-point inclusion probabilities}

Let $(S_t^x)_{t\geq0}$ be the ordinary evolving set started from $\{x\}$. Taking
$D=V$ in Lemma~\ref{lem:punctured-intertwining}, we obtain, for $x,y\in V$ and $t\geq0$,
\begin{equation}\label{eq:heat-shadow}
  \Pp_{\{x\}}(y\in S_t^x)=\frac{\pi(x)}{\pi(y)}P_t(x,y).
\end{equation}
Multiply Equality~\eqref{eq:heat-shadow} by $\pi(y)f(y)$ and sum over $y$.
For every $f:V\to\mathbb R$ this gives
\[
  \frac{1}{\pi(x)}\Ee_{\{x\}}\!\left[\sum_{y\in S_t^x}\pi(y)f(y)\right]=P_tf(x).
\]

\subsection{Averaging at an exponential time}

Let us integrate Equality~\eqref{eq:heat-shadow} against $qe^{-qt}\dd t$.
Since $P_t=e^{-tL}$ and $T_q$ is independent of the evolving set,
\[
  \Pp_{\{x\}}(y\in S_{T_q}^x)=\frac{\pi(x)}{\pi(y)}K_q(x,y),
  \qquad K_q=q(qI+L)^{-1}.
\]
Consequently, for every $f:V\to\mathbb R$,
\[
  \frac{1}{\pi(x)}\Ee_{\{x\}}\!\left[\sum_{y\in S_{T_q}^x}\pi(y)f(y)\right]
  =K_q f(x).
\]
Taking $f=\1_{\{y\}}$ and applying Equality~\eqref{eq:intro-root-law}, we recover
$\Pp(\rho_q(x)=y)=K_q(x,y)$. In particular $\Pp(x\in\Root)=K_q(x,x)$.
For several prescribed roots, the determinantal formula gives
$\Pp(A\subseteq\Root)=\det(K_q(a,b))_{a,b\in A}$. Let us see how the successive
diagonal factors in Proposition~\ref{prop:determinantal-comparison} appear when
the domain is punctured.

\subsection{Punctured resolvents and hitting}

Fix a non-empty $D\subseteq V$ and recall the Dirichlet operator $L_D$ from
\eqref{eq:dirichlet-generator-definition}. The semigroup of the walk killed on
leaving $D$ is $e^{-tL_D}$. Integrating Equality~\eqref{eq:singleton-intertwining}
now gives, for $x,y\in D$,
\begin{equation}\label{eq:punctured-resolvent}
  \Pp^D_{\{x\}}(y\in S_{T_q}^D)
  =\frac{\pi(x)}{\pi(y)}q(qI+L_D)^{-1}(x,y).
\end{equation}
On setting $x=y$ we obtain the factor associated with a prescribed root in $D$.
If we multiply by $\pi(y)/\pi(x)$ and sum over $y\in D$, we obtain
\[
 \frac{1}{\pi(x)}\Ee^D_{\{x\}}[\pi(S_{T_q}^D)]
   =\big[q(qI+L_D)^{-1}\1_D\big](x).
\]
By Equality~\eqref{eq:exit-projection} the probability of reaching $D^c$ before
killing is therefore
\[
 \Pp_x(\tau_{D^c}\le T_q)=1-\big[q(qI+L_D)^{-1}\1_D\big](x).
\]

\subsection{Occupation and exploration time}

Let $x\in V$ and $A\subseteq V$. Write $O_q(x,A)$ for the expected time spent in
$A$ by the walk from $x$ before killing. Then
\begin{equation}\label{eq:occupation-es}
 O_q(x,A)=\frac{1}{\pi(x)}\int_0^\infty e^{-qt}
                  \Ee_{\{x\}}[\pi(S_t^x\cap A)]\dd t.
\end{equation}
Indeed, by Tonelli's theorem and the independence of the clock,
\[
 O_q(x,A)=\Ee_x\!\left[\int_0^{T_q}\1_{\{X_t\in A\}}\dd t\right]
         =\int_0^\infty e^{-qt}\Pp_x(X_t\in A)\dd t.
\]
Summing Equality~\eqref{eq:heat-shadow} against $\pi(y)$ over $y\in A$ gives
$\Ee_{\{x\}}[\pi(S_t^x\cap A)]=\pi(x)\Pp_x(X_t\in A)$, which proves
\eqref{eq:occupation-es}.

Finally let $a\in D\subseteq V$ and put $\sigma_D=T_q\wedge\tau_{D^c}$.
The same calculation with the killed semigroup gives
\[
  \Ee_a[\sigma_D]=\frac{1}{\pi(a)}\int_0^\infty e^{-qt}
     \Ee^D_{\{a\}}[\pi(S_t^D)]\dd t.
\]
In Wilson's algorithm, take $D$ to be the complement of the forest already built.
Then $\sigma_D$ is the time the walk from $a$ runs before it is killed or attaches
to that forest.

\paragraph*{Acknowledgements.} The author thanks S\'ebastien Darses and Erwan Hillion for helpful discussions and comments.

\end{document}